\documentclass[12pt, twoside, leqno]{article}
\usepackage{amsmath,amsthm,amscd}
\usepackage{amssymb}
\usepackage{enumerate}
\usepackage{amsthm}
\newtheorem{thm}{Theorem}[section]

\newtheorem{lem}[thm]{Lemma}
\newtheorem{prop}[thm]{Proposition}
\newtheorem{rem}[thm]{Remark}
\theoremstyle{definition}

\newcommand{\Rat}{{\mathbb{Q}}}
\newcommand{\Real}{{\mathbb{R}}}
\newcommand{\Com}{{\mathbb{C}}}
\newcommand{\Zet}{{\mathbb{Z}}}
\newcommand{\Nat}{{\mathbb{N}}}

\newcommand{\field}{{\mathbb{K}}}
\newcommand{\Pro}{{\mathbb{P}}}

\newcommand{\Ratcl}{{\overline{\mathbb{Q}}}}

\newcommand{\Gl}{{\mathrm{GL}}}

\newcommand{\h}{{\mathrm{h}}}
\newcommand{\ha}{{\hat{\mathrm{h}}}}
\newcommand{\hi}{{\mathrm{h}'}}

\newcommand{\Lie}{{\mathrm{Lie}}}

\newcommand{\ib}{{\mathbf{i}}}

\newcommand{\Spec}{{\mathrm{Spec}}}

\newcommand{\gfr}{{\mathfrak{g}}}
\newcommand{\hfr}{{\mathfrak{h}}}
\newcommand{\Ac}{{\mathcal{A}}}

\newcommand{\Bc}{{\mathcal{B}}}

\newcommand{\Tr}{{\mathrm{Tr}}}
\newcommand{\eb}{{\mathrm{\mathbf{e}}}}
\newcommand{\fb}{{\mathrm{\mathbf{f}}}}
\newcommand{\ub}{{\mathrm{\mathbf{u}}}}
\newcommand{\vb}{{\mathrm{\mathbf{v}}}}
\newcommand{\wb}{{\mathrm{\mathbf{w}}}}
\newcommand{\jb}{{\mathrm{\mathbf{j}}}}
\newcommand{\dist}{{\mathrm{d}}}

\numberwithin{equation}{section}

\begin{document}
\setlength{\parindent}{0pt}
\baselineskip=17pt

\title{Linear independence measure of logarithms over affine  groups}

\author{Mario  Huicochea\\
Facultad de Ciencias, UNAM\\ 
E-mail: dym@cimat.mx
}

\date{}

\maketitle
\begin{abstract}
 Linear forms in logarithms over connected commutative  algebraic groups over $\Ratcl$ have been studied widely. However, the theory of  linear forms in logarithms over noncommutative  algebraic groups have not been developed as the one of the  commutative algebraic groups and in this paper we start studying linear forms in logarithms over affine groups. Let $G$ be a connected affine group over $\Ratcl$ with Lie algebra $\gfr$ and let $\Bc$ be a fixed basis of $\gfr$.  Let  $W$ be a linear subspace of $\gfr\otimes_\Ratcl\Com=\Lie(G(\Com))$  defined by $\Ratcl-$linear forms  and we denote by $\h(W)$ its height. For any $\ub\in \Lie(G(\Com))$,  denote by $\dist(\ub,W)$ the distance (with respect to $\Bc$) between $\ub$ and $W$. In this paper we show that  if  $\ub\in\Lie(G(\Com))\setminus W$ is such that $\exp_{G(\Com)}(\ub)\in G(\Ratcl)$, then there are $b$ independent of $W$ and  $c$ independent of $W$ and $\ub$ such that 
 \begin{equation*}
 \log(\dist(\ub,W))>-bc\max\{\h(W),1\}
 \end{equation*}
Moreover, $b$ and $c$ are effective and  completely explicit.
\end{abstract}
\section{Introduction}
The theory of linear forms in logarithms has been one of the most fruitful fields in number theory for many years. The purpose of this paper is to find lower bounds of   linear forms in logarithms over  affine groups over $\Ratcl$. Moreover, the result we shall obtain will be completely explicit and optimal in the sense of  Remark \ref{R10}.

 Let $\h:\Pro^N(\Ratcl)\rightarrow \Real$ be the absolute logarithmic function. In 1935 Gelfond showed that there is a constant $b>0$ with the following property:  for all  $\beta_1,\beta_2\in\Ratcl$ and  $\alpha_1,\alpha_2\in\Ratcl\setminus\{0\}$ such that $\frac{\log(\alpha_1)}{\log(\alpha_2)}\not\in\Rat$, there is $\kappa>5$   independent of $\{\beta_1,\beta_2\}$  such that
\begin{equation*}
\log(|\beta_1\log(\alpha_1)+\beta_2\log(\alpha_2)|)\geq-b\max\{\h([\beta_1:\beta_2]),1\}^\kappa.
\end{equation*} 
 Later Baker could generalize the previous result for several variables. This means that there is $b>0$ such that for all 
$\beta_0,\ldots,\beta_n\in\Ratcl$ and  $\alpha_1,\ldots,\alpha_n\in\Ratcl\setminus\{0\}$ with
\begin{equation*}
\beta_0+\sum_{i=1}^n\beta_i\log(\alpha_i)\neq 0
\end{equation*}
there exists a constant $\kappa$ enoughly big and independent of $\h([\beta_0:\ldots:\beta_n])$ such that
\begin{equation*}
\log\bigg(\bigg|\beta_0+\sum_{i=1}^n\beta_i\log(\alpha_i)\bigg|\bigg)>-b\max\{\h([\beta_0:\ldots:\beta_n]),1\}^\kappa.
\end{equation*}
In the following years, Baker \emph{et al.} were able to improve the values of $b$ and  $\kappa$  and they  obtained several applications, see  \cite{Wustholz02}. The last significant progress in this direction  was done by Matveev in \cite{Matveev}. The previous results are minorizations of linear forms in logarithms   over connected algebraic tori over $\Ratcl$. Thus it was natural to study linear forms in logarithms   over other  algebraic groups. Let $G$  be a connected algebraic group over $\Ratcl$ of dimension $n$. We denote by $\gfr$ its Lie algebra  and we  fix a basis $\Bc=\{\eb_1,\ldots,\eb_n\}$ of $\gfr$. We identify $G(\Ratcl)$ with the closed points of $G$ and we consider it a subset of $G(\Com)$.   The set $G(\Com)$ has a $\Com-$Lie group structure such that there exists an isomorphism of $\Com-$vector spaces $\Lie(G(\Com))\cong \gfr\otimes_\Ratcl \Com$; thus we identify $\Lie(G(\Com))$ with $\gfr\otimes_\Ratcl \Com$ and assume that $\Bc$  is a basis of $\Lie(G(\Com))$ (identifying $\gfr$ with $\gfr\otimes_\Ratcl 1$). Moreover, we define the norm
\begin{equation*}
\|\cdot\|_\Bc:\Lie(G(\Com))\longrightarrow\Real,\qquad \bigg\|\sum_{i=1}^nz_i\eb_i\bigg\|_\Bc=\sum_{i=1}^n|z_i|^2
\end{equation*}
and for any subset $V$ of $\Lie(G(\Com))$ and $\ub\in \Lie(G(\Com))$
\begin{equation*}
\dist_\Bc(\ub,V):=\inf_{\vb\in V}\|\ub-\vb\|_\Bc.
\end{equation*}
 For any field $\Rat\subseteq \field\subseteq \Com$, we say that  a linear subspace $W$ of $\Lie(G(\Com))$  is defined over $\field$ with respect to $\Bc$ if it has a basis $\{\wb_1,\ldots,\wb_d\}$ where $\wb_j:=\sum_{i=1}^ny_{i,j}\eb_i$ with $y_{i,j}\in\field$ for all $i\in\{1,\ldots,n\}$ and $j\in\{1,\ldots,d\}$. When $W$ is defined over $\Ratcl$, we denote by $\h_\Bc(W)$ the height of $W$ with respect to $\Bc$ (in \emph{Notation and conventions} we state the precise definition of $\h_\Bc(W)$). Assume that $G$ is embedded in $\Pro^N$ and set $e:=\exp_\Com(1)$.  First, Philippon and Waldschmidt \cite[Thm. 1.2]{Philippon-Waldschmidt88} showed that if $G$ is commutative, $W$ is an hyperplane of $\Lie(G(\Com))$  defined over $\Ratcl$ with respect to $\Bc$ and $\ub\in\Lie(G(\Com))\setminus W$ is such that $\exp_{G(\Com)}(\ub)\in G(\Ratcl)$, then there are
\begin{equation*}
b_1:=\log\Big(\max\big\{e,\h(\exp_{G(\Com)}(\ub)),\|u\|^2_\Bc\big\}\Big)^{n+1}\cdot\max\big\{e,\h(\exp_{G(\Com)}(\ub)),\|\ub\|^2_\Bc\big\}^{n}
\end{equation*} 
 and  $c_1$ independent of $\ub$ and $W$ such that  
 \begin{equation*}
\log(\dist_\Bc(\ub,W))\geq -c_1b_1\max\{e,\h_\Bc(W)\}^{n+1}.
\end{equation*}
 Later   Hirata-Kohno  \cite[Thm. 2.1]{Hirata-Kohno} showed that with the assumptions as in the statement of Philippon and Waldschmidt there is  $c_2$ independent of $\ub$ and $W$ such that  
 \begin{equation*}
\log(\dist_\Bc(\ub,W))\geq -c_2b_1\log\big(\max\{e,\h_\Bc(W)\}\big)\max\{e,\h_\Bc(W)\}.
\end{equation*}
  Finally, Gaudron improved in \cite[Thm. 2]{Gaudron05} the results of Hirata-Kohno. This means that with the assumptions of  the statement of Philippon and Waldschmidt there is $c_3$ independent of $\ub$ and $W$ such that  
 \begin{equation*}
\log(\dist_\Bc(\ub,W))\geq -c_3b_1\max\{1,\h_\Bc(W)\}.
\end{equation*}
Moreover, Gaudron showed that this inequality is optimal in terms of $\h_\Bc(W)$, see \cite[Sec. 3.2]{Gaudron05}. The constants $c_i$ obtained by  Philippon, Waldschmidt, Hirata Kohno and Gaudron are effective but they are not completely explicit.  The few exceptions are the cases where $G$ is a connected commutative affine group (see \cite[Thm.]{Baker-Wustholz} and \cite[Thm. 9.1]{Waldschmidt}),  $G$ is the direct product of connected commutative elliptic curves (see \cite[Thm. 2.1]{David}), and $G$ is an abelian variety (with some technical conditions, see  \cite[Thm. 1.1]{Gaudron06}).

Hitherto we have seen results about linear forms in logarithms  over commutative connected  algebraic groups. Hence it is natural to study linear forms in logarithms  over noncommutative connected  algebraic groups. We give the first step in this way finding lower bounds for linear forms in logarithms  over  connected affine  groups over $\Ratcl$.  Let $\field$ be a finite extension of $\Rat$ and  $G$ be a connected affine group over $\field$. Let $\gfr$ be the Lie algebra of $G$ and fix a basis $\Bc=\{\eb_1,\ldots,\eb_n\}$ of $\gfr$. It is a well known that affine groups are linear, see \cite{Borel}. Thus we assume  that $G$ is an algebraic subgroup of $\Gl_m$ and  we embed $\Gl_m$ into $\Pro^{m^2}$ as follows:
\begin{equation*}
 \Gl_m\longrightarrow \Pro^{m^2},\qquad \Big((g_{i,j})_{1\leq i,j\leq n}\Big)\mapsto\big[1:g_{1,1}:g_{1,2}:\ldots:g_{1,m}:g_{2,1}:\ldots:g_{m,m}\big].
\end{equation*}
 Let $\mathfrak{gl}_m$ be the Lie algebra of $\Gl_m$ and recall that we identify $\Lie(\Gl_m(\Com))$ with $\mathfrak{gl}_m\otimes_\field\Com$. The Lie algebra $\mathfrak{gl}_m$ will be identified with the set of $m\times m$ matrices with coefficients in $\field$ from now on. For each $i,j\in\{1,\ldots,m\}$, let $\fb_{(i-1)m+j}$ be the matrix with all entries equal to $0$ except the one in the ith row and jth column which is $1$. Hence $\Bc_0:=\big\{\fb_{(i-1)m+j}:\;i,j\in\{1,\ldots,m\}\big\}$ is a basis of $\mathfrak{gl}_m$ and consequently a basis of $\Lie(\Gl_m(\Com))$. Denote by $\star$ the martix product in $\Lie(\Gl_m(\Com))$. For each $\ub\in \Lie(G(\Com))$, let $\jb_\ub$ be the Jordan canonical form of $\ub$ and let $C_\ub(\field)$ be the subset of  invertible elements $\vb$ of $\Lie(\Gl_m(\Com))$ (in other words $\vb\in\Gl_m(\Com)$ is such that $\det(\vb)\neq 0$) with entries in $\field$ such that $\vb\star \jb_\ub\star \vb^{-1}=\ub$.  If $\vb=\sum_{k=1}^{m^2}z_k\fb_k$ with $z_1,\ldots,z_{m^2}\in\Ratcl$, write $\h'(\vb):=\h([1:z_1:\ldots:z_{m^2}])$.  We say that $\ub$ is a $\field-$point if $\exp_{G(\Com)}(\ub)\in G(\field)$ and the characteristic polynomial of $\exp_{G(\Com)}(\ub)$ (considering $\exp_{G(\Com)}(\ub)$  an element of $\Gl_m(\Com)$) splits in $\field$. If $\ub$ is a $\field-$point, then $C_\ub(\field)$ is not empty, see \cite[Sec. 7.3]{Hoffman-Kunze}. For $\ub$ which is a $\field-$point, we define 
   \begin{equation*}
 b_2(\ub):= \inf_{\vb\in C_\ub(\field)}\max\bigg\{e,\hi(\vb),\frac{\|\vb\|^m_{\Bc_{0}}}{\det(\vb)}\bigg\} 
 \end{equation*}
 For all  $d\in\{1,\ldots,n\}$, set
\begin{align*}
\Ac_d&:=\big\{\ib=(i_1,\ldots,i_d)\in\Nat^d:\;1\leq i_1<\ldots<i_d\leq n\big\}\\
\Ac'_d&:=\big\{\ib=(i_1,\ldots,i_d)\in\Nat^d:\;1\leq i_1<\ldots<i_d\leq m^2\big\}
\end{align*}
and denote by $\Delta_\field$ the discriminant of $\field$. The main result of this paper is the following.
\begin{thm}
\label{R1}
Let $d\in\{1\ldots,n-1\}$ and $W$ be a $d-$dimensional linear subspace of  $\Lie(G(\Com))$ defined over $\field$ with respect to $\Bc$. Take  $\ub\in\Lie(G(\Com))\setminus W$  a $\field-$point. For each $j\in\{1,\ldots, n\}$, write  $\eb_j=\sum_{i=1}^{m^2}z_{i,j}\fb_i$. For all $\ib=(i_1,\ldots,i_d)\in \Ac_d$ and $\ib'=(i'_1,\ldots,i'_d)\in \Ac'_d$, set
 \begin{equation*}
 \Gamma_{\ib,\ib'}:=\det\Big(\big(z_{i_k,i'_j}\big)_{1\leq k,j\leq d}\Big)
 \end{equation*} and let $\Gamma$ the element of $\Pro^{|\Ac_d||\Ac'_d|-1}(\field)$ with entries $\Gamma_{\ib,\ib'}$ for each $\ib\in\Ac_d$ and $ \ib'\in\Ac'_d$. Set
 \begin{align*}
 b_3:=&b_2(\ub)\max\big\{e,\h\big(\exp_{G(\Com)}(\ub)\big),\|\ub\|_\Bc\big\},\\
 c_4:=&2^{32m+31}\max\big\{\log(|\Delta_\field|),1\big\}[\field:\Rat]^{m+5}(m^2-d)^4m^{m^2+8m+25}\\
 &\cdot\max\Bigg\{e,\sqrt{n\sum_{i=1}^{m^2}\sum_{j=1}^{n} |z_{i,j}|^2}\Bigg\}^{m^2+2m+2}\max\{1,\h(\Gamma)\}.
 \end{align*}
 Then
\begin{equation*}
\log(\dist_\Bc(\ub,W))\geq -c_4\log(b_3)b^{m+1}_3\max\{1,\h_\Bc(W)\}.
\end{equation*}
\end{thm}
It is clear that if $\ub\in \Lie(G(\Com))$ and $\exp_{G(\Com)}(\ub)\in G(\Ratcl)$, then $\field$ may be taken big enough to have that $\ub$ is a $\field-$point, see Remark \ref{R9}.  With respect to $\h_\Bc(W)$, Theorem \ref{R1} is optimal as we will see in Remark \ref{R10}.  The number $b_2(\ub)$ may look unnatural in Theorem \ref{R1} (i.e. it has not analogous value in the commutative case) but, as it is seen in Remark \ref{R11}, it is necessary. 

This paper is organized as follows. In Section 2 we state the definition of the  height of a linear subspace and  we present some technical lemmas that will be used in the subsequent sections. In Section 3 we obtain a similar result to Theorem \ref{R1} when  $G=\Gl_m$ and $W$ is a hyperplane of $\Lie(\Gl_m(\Com))$. In Section 4  we use the  main statement of the previous section to  demonstrate  Theorem \ref{R1}; after the demonstration is concluded,   we make some remarks about this theorem. 
\subsection*{Notation and conventions}
In this paper, $\exp_{\Com}$ is the complex exponential function while $e:=\exp_\Com(1)$. For any $z\in\Com$, $\overline{z}$ is its complex conjugate and    $|z|:=\sqrt{z\overline{z}}$. We assume that $\Ratcl$, $\Rat$ and all its finite extensions are already embedded in $\Com$ in this article.

Let $\field$ be a finite extension of $\Rat$. Denote by $\Ac(\field)_\infty$  the set of archimedean places of $\field$ and $\Ac(\field)$  the set of  places of $\field$. For all $v\in \Ac(\field)_\infty$, let $|\cdot |_v$ be the normalized absolute value of the completion  $\field_v$ of $\field$ in the place $v$ (i.e. $|2|_v=2$). Set $n_v=2$ if $v$ is complex and $n_v$ if $v$ is not complex. Let $v\in \Ac(\field)\setminus \Ac(\field)_\infty$ be $p-$adic (an extension of the place $|\cdot|_p$ of $\Rat$) and $|\cdot |_v$ is the normalized absolute value of the completion  $\field_v$ of $\field$ in the place $v$ (i.e. $|p|_v=\frac{1}{p}$); in this case, write $n_v:=[\field_v:\Rat_p]$. For all $p=[p_0:\ldots:p_N]\in\Pro^N(\field)$, define the logarithmic heights as follows
\begin{align*}
\h(p):=&\frac{1}{[\field:\Rat]}\sum_{v\in\Ac(\field)}n_v\log\big(\max\{|p_0|_v,\ldots,|p_N|_v\}\big),\\
\hi(p):=&\frac{1}{[\field:\Rat]}\sum_{v\in\Ac(\field)}n_v\log\big(\max\{1,|p_0|_v,\ldots,|p_N|_v\}\big),\\
\ha(p):=&\frac{1}{[\field:\Rat]}\sum_{v\in\Ac(\field)\setminus\Ac(\field)_\infty}n_v\log\big(\max\{|p_0|_v,\ldots,|p_N|_v\}\big)\\
&+\frac{1}{[\field:\Rat]}\sum_{v\in\Ac(\field)_\infty}\frac{n_v}{2}\log\Bigg(\sum_{k=0}^N|p_k|^2_v\Bigg).
\end{align*}
Note that
\begin{equation*}
\h(p)\leq\ha(p) \leq\h(p)+\frac{\log(N+1)}{2}. 
\end{equation*}
The definitions of $\h(p),\ha(p)$ and $\hi(p)$ do not depend on $\field$ and therefore their definitions may be extended to $\Pro^N(\Ratcl)$.  For all $x\in\Ratcl$, we write $\hi(x):=\h([x:1])$.

  All the varieties $X$ considered in this paper  will be $\Spec(\Com)$-schemes. We say that $X$ is defined over $\field\subseteq \Com$ if there is $Y\rightarrow \Spec(\field)$ such that $X=Y\times_{\Spec(\field)}\Spec(\Com)$.  If $X$ is defined over $\field$, the set of closed points  of $X\rightarrow \Spec(\field)$ is identified with 
\begin{equation*}
X(\field):=\mathrm{Hom}_{\Spec(\field)}(\Spec(\field),X)
\end{equation*}
and therefore for any fields $\Rat\subseteq \field\subseteq\mathbb{L}\subseteq \Com$ we shall consider $X(\field)\subseteq X(\mathbb{L})$. For any algebraic group $H$ over $\field$ of dimension $n$,  the group $H(\Com)$ has a $\Com-$Lie group structure. The Lie algebra associated to $H(\Com)$ will be denoted  by  $\Lie(H(\Com))$ and  its exponential map  will be denoted by $\exp_{H(\Com)}:\Lie(H(\Com))\rightarrow H(\Com)$. If $\hfr$ is the Lie algebra of $H$, then $\Lie(H(\Com))$  will be identified with $\hfr\otimes_\field \Com$. Thus given a basis $\Bc=\{\eb_1,\ldots, \eb_n\}$ of $\hfr$, we abuse of notation considering it a basis of $\Lie(H(\Com))$. For any $\vb_1,\ldots\vb_r\in\Lie(H(\Com))$, denote by $\big< \vb_1,\ldots,\vb_r\big>$ the subspace of $\Lie(H(\Com))$ generated by $\vb_1,\ldots\vb_r$.  We define the norm 
\begin{equation*}
\|\cdot\|_\Bc:\Lie(H(\Com))\longrightarrow\Real,\qquad \bigg\|\sum_{i=1}^nz_i\eb_i\bigg\|_\Bc=\sum_{i=1}^n|z_i|^2
\end{equation*}
and for any subset $V$ of $\Lie(H(\Com))$ and $\ub\in \Lie(H(\Com))$
\begin{equation*}
\dist_\Bc(\ub,X):=\inf_{\vb\in V}\|\ub-\vb\|_\Bc.
\end{equation*}
We say that  a linear subspace $W$ of $\Lie(H(\Com))$  is defined over $\field$ with respect to $\Bc$ if it has a basis $\{\wb_1,\ldots,\wb_d\}$ where $\wb_j:=\sum_{i=1}^ny_{i,j}\eb_i$ with $y_{i,j}\in\field$ for all $i\in\{1,\ldots,n\}$ and $j\in\{1,\ldots,d\}$; equivalently   $W$ is defined over $\field$ with respect to $\Bc$ if $W$ is defined by $\field-$linear forms with respect to $\Bc$. Now we define $\h_\Bc(W)$ whenever $W$ is defined over $\Ratcl$ with respect to $\Bc$. Set
\begin{equation*}
\Ac_d:=\big\{\ib=(i_1,\ldots,i_d)\in\Nat^d:\;1\leq i_1<\ldots<i_d\leq n\big\}
\end{equation*}
 and for each $\ib\in \Ac_d$ write
 \begin{equation*}
 \Lambda_\ib:=\det\Big(\big(y_{i_k,j}\big)_{1\leq k,j\leq d}\Big).
 \end{equation*}
Considering $\big[\Lambda_\ib\big]_{\ib\in\Ac_d}$ as an element of $\Pro^{|\Ac_d|-1}(\Ratcl)$,  write
 \begin{equation*}
\h_\Bc(W):=\h\Big(\big[\Lambda_\ib\big]_{\ib\in\Ac_d}\Big)\qquad\text{and}\qquad \ha_\Bc(W):=\ha\Big(\big[\Lambda_\ib\big]_{\ib\in\Ac_d}\Big).
 \end{equation*}
The definitions of $\h_\Bc(W)$ and $\ha_\Bc(W)$ depend on $\Bc$, however it can be proven that they do not depend on $\{\wb_1,\ldots,\wb_d\}$, see \cite{Schmidt}.  In particular, if  $\wb:=\sum_{j=1}^s z_i\eb_i\neq 0$ with $z_1,\ldots,z_s\in\Ratcl$, then
\begin{equation*}
\h_\Bc\big(\big<\wb\big>\big)=\h([z_1:\ldots:z_s])\qquad\text{and}\qquad \ha_\Bc\big(\big<\wb\big>\big)=\ha([z_1:\ldots:z_s]).
\end{equation*}

In this paper $\field$ is a  finite extension  of $\Rat$ and  $G$ is a $n-$dimensional connected algebraic subgroup of $\Gl_m$ defined over  $\field$ with $\gfr$ its associated Lie algebra. We denote by $\mathfrak{gl}_m$  the Lie algebra of $\Gl_m$ which be identified with the set of $m\times m$-matrices with coefficients in $\field$.  For each $i,j\in\{1,\ldots,m\}$, let $\fb_{(i-1)m+j}$ be the matrix with all entries equal to $0$ except the one in the ith row and jth column which is $1$. Hence $\Bc_0:=\big\{\fb_{(i-1)m+j}:\;i,j\in\{1,\ldots,m\}\big\}$ is a basis of $\mathfrak{gl}_m$ and therefore a basis of $\Lie(\Gl_m(\Com))$. For  $z_1\ldots,z_{m^2}\in\Ratcl$, write
\begin{equation*}
\hi\bigg(\sum_{k=1}^{m^2}z_k\fb_k\bigg):=\hi([z_1:\ldots:z_{m^2}]).
\end{equation*}
  $\Gl_m$ is embedded into $\Pro^{m^2}$ with the morphism
\begin{equation*}
 \Gl_m\longrightarrow \Pro^{m^2},\qquad \Big((g_{i,j})_{1\leq i,j\leq m}\Big)\mapsto\big[1:g_{1,1}:g_{1,2}:\ldots:g_{1,m}:g_{2,1}:\ldots:g_{m,m}\big].
\end{equation*}
Given $\ub\in\Lie(G(\Com))$ and $W$ a linear subspace of $\Lie(G(\Com))$, $c_i$ will a real number independent of $\ub$ and $W$, and $b_i$ will be a real number independent of $W$ for all $i\in\Nat$.

  \section{Preliminaries} 
 In this section we demonstrate some auxiliary lemmas that will be used in the next sections. From now on $\Bc:=\{\eb_1,\ldots,\eb_n\}$ is a basis of $\gfr$ and therefore a basis of $\Lie(G(\Com))$. Since $G$ is an algebraic subgroup of $\Gl_m$, we shall consider $\gfr$ (resp. $\Lie(G(\Com))$) a Lie subalgebra of $\mathfrak{gl}_m$ (resp. $\Lie(\Gl_m(\Com))$). Thus  there are $z_{i,j}\in\field$  such that  $\eb_j=\sum_{i=1}^{m^2}z_{i,j}\fb_i$ for each  $j\in\{1,\ldots,n\}$.  Let $W$ be a linear subspace of $\Lie(G(\Com))$ defined over $\field$ with respect to $\Bc$. Since $\gfr$ is a Lie subalgebra of $\mathfrak{gl}_m$, $W$ may be considered a  linear subspace of $\Lie(G(\Com))$ defined over $\field$ with respect to $\Bc_0$.

 \begin{lem}
 \label{R2} Let 
\begin{equation*}
W^\perp:=\Bigg\{\sum_{k=1}^{m^2}y_k\fb_k\in\Lie(\Gl_m(\Com)):\;\sum_{k=1}^{m^2}y_k\overline{z_k}=0 \text{ whenever }\sum_{k=1}^{m^2}z_k\fb_k\in W\Bigg\}.
\end{equation*}
Then
   \begin{equation*}
 \big|\h_{\Bc_0}(W)-\h_{\Bc_0}\big(W^\perp\big)\big|\leq \log\binom{m^2}{d}.
 \end{equation*}
 \end{lem}
\begin{proof}
 From \cite[Lemma 5G \& Sec. 8]{Schmidt} it its deduced that 
\begin{equation*}
\ha_{\Bc_0}(W)=\ha_{\Bc_0}\big(W^\perp\big).
\end{equation*}
Recall that for all for all $p\in\Pro^N(\Ratcl)$
\begin{equation*}
\h(p)\leq\ha(p) \leq\h(p)+\frac{\log(N+1)}{2}. 
\end{equation*}
Thus 
\begin{equation*}
\big|\h_{\Bc_0}(W)-\ha_{\Bc_0}(W)\big|,\big|\ha_{\Bc_0}\big(W^\perp\big)-\h_{\Bc_0}\big(W^\perp\big)\big|\leq \frac{1}{2}\log\binom{m^2}{d}
\end{equation*}
and the result follows easily.
\end{proof}
Now we compare the height of $W$ with respect to $\Bc$ and the height of $W$ with respect to $\Bc_0$ (as a linear subspace of $\Lie(\Gl_m(\Com))$). For all  $d\in\{1,\ldots,n\}$, set
\begin{align*}
\Ac_d&:=\big\{\ib=(i_1,\ldots,i_d)\in\Nat^d:\;1\leq i_1<\ldots<i_d\leq n\big\}\\
\Ac'_d&:=\big\{\ib=(i_1,\ldots,i_d)\in\Nat^d:\;1\leq i_1<\ldots<i_d\leq m^2\big\}
\end{align*}
and for each $\ib=(i_1,\ldots,i_d)\in \Ac_d,\ib'=(i'_1,\ldots,i'_d)\in \Ac'_d$ write
 \begin{equation*}
 \Gamma_{\ib',\ib}:=\det\Big(\big(z_{i'_k,i_j}\big)_{1\leq k,j\leq d}\Big).
 \end{equation*}
Let $\Gamma$ be the element of $\Pro^{|\Ac_d||\Ac'_d|-1}(\field)$ with entries $\Gamma_{\ib,\ib'}$. 
\begin{lem}
\label{R3}
\begin{equation*}
\h_{\Bc_0}(W)\leq \h_{\Bc}(W)+\h(\Gamma)+\log\binom{m^2}{d}.
\end{equation*}
\end{lem}
\begin{proof}
Let $\{\wb_1,\ldots,\wb_d\}$  be a  basis of $W$ such that  $\wb_j:=\sum_{i=1}^ny_{i,j}\eb_i$ with $y_{i,j}\in\field$ for all $i\in\{1,\ldots,n\}$ and $j\in\{1,\ldots,d\}$.  For each $i\in\{1,\ldots,m^2\}$ and $j\in\{1,\ldots,d\}$, define $y'_{i,j}:=\sum_{k=1}^nz_{i,k}y_{k,j}$. Hence, for all $j\in\{1,\ldots,d\}$
\begin{equation}
\label{E1}
\wb_j=\sum_{i=1}^{m^2}y'_{i,j}\fb_i.
\end{equation}
Now set for all $\ib=(i_1,\ldots,i_d)\in\Ac_d$ and $\ib'=(i'_1,\ldots,i'_d)\in\Ac'_d$
\begin{align*}
 \Lambda_{\ib}:=\det\Big(\big(y_{i_k,j}\big)_{1\leq k,j\leq d}\Big)\qquad\text{and}\qquad \Lambda'_{\ib'}:=\det\Big(\big(y'_{i'_k,j}\big)_{1\leq k,j\leq d}\Big).
 \end{align*}
 For all $\ib'\in\Ac'_d$, the Cauchy-Binet formula and (\ref{E1}) yield 
 \begin{equation*}
 \Lambda'_{\ib'}=\sum_{\ib\in\Ac_d}\Gamma_{\ib',\ib}\cdot \Lambda_{\ib}
 \end{equation*}
 Thus a standard height bound calculation (see for instance \cite[Thm. B.2.5]{Hindry-Silverman}) leads to
 \begin{equation*}
 \h_{\Bc_0}(W)\leq \h_{\Bc}(W)+\h(\Gamma)+\log\binom{m^2}{d}.
 \end{equation*}
 \end{proof}
\begin{lem}
\label{R4}
Let $g,l\in \Gl_m(\Ratcl)$. Then
\begin{equation*}
\h\big(l^{-1}gl\big)\leq \h(g)+ m\h(l)+\log\binom{m^2+m}{m}+\log(m^2+1).
\end{equation*}
\end{lem}
\begin{proof}
Write $g=[1:g_1:\ldots:g_{m^2}]$, $l=[1:l_1:\ldots:l_{m^2}]$ and $l^{-1}=\big[1:l'_1:\ldots:l'_{m^2}\big]$. Define the polynomials $P_1,\ldots,P_{m^2}\in \Ratcl[x_0,\ldots,x_{m^2}]$ as follows: if $k=(i-1)m+j$ with  $1\leq i,j\leq m$, then
\begin{equation*}
P_k(x_0,\ldots x_{m^2})=\sum_{1\leq r,s\leq m}l'_{(i-1)m+r}l_{(s-1)m+j}x_{(r-1)m+s}.
\end{equation*}
 Thus, if $P:=(1:P_1:\ldots:P_{m^2})$ 
\begin{equation}
\label{E2}
\h\big(l^{-1}gl\big)=\h(P(1,g_1,\ldots, g_{m^2})).
\end{equation}
Considering $\big[l'_il_j\big]_{1\leq i,j\leq m^2}$ as an element of $\Pro^{m^4-1}$, a standard height bound calculation (see for instance \cite[Thm. B.2.5]{Hindry-Silverman}) yields 
\begin{align}
\label{E3}
\h(P(1,g_1,\ldots, g_{m^2}))&\leq
\h(g)+\hi\Big(\big[l'_il_j\big]_{1\leq i,j\leq m^2}\Big)+\log(m^2+1).
\end{align}
The scalars $l'_1,\ldots,l'_{m^2}$ are the cofactors of $l$ divided by $\det(l)$.  Thus there are homogeneous polynomials $Q_1,\ldots,Q_{m^2}\in \Zet[x_1,\ldots,x_{m^2}]$ with coefficients $\pm 1$ and degree $m-1$ such that
\begin{equation*}
\frac{Q_k(l_1,\ldots, l_{m^2})}{\det(l)}=l'_k.
\end{equation*}
Proceeding ones more as in \cite[Thm. B.2.5]{Hindry-Silverman}, we find that
\begin{align}
\label{E4}
\hi\Big(\big[l'_il_j\big]_{1\leq i,j\leq m^2}\Big)&\leq  m\h(l)+\log\binom{m^2+m}{m}
\end{align}
 and the statement follows from (\ref{E2}), (\ref{E3}) and (\ref{E4}).
\end{proof}
The proof of the following lemma is exactly the same as the proof of Lemma \ref{R4}.
\begin{lem}
\label{R5}
Let $\ub,\vb\in \Lie(\Gl_m(\Com))$ with $\vb$ an invertible matrix and let $\vb^{-1}$ be the inverse matrix of $\vb$.If the entries of $\ub$ and $\vb$ are in $\Ratcl$,   then
\begin{equation*}
\hi\big(\vb^{-1}\star \ub\star\vb\big)\leq \hi(\ub)+ m\hi(\vb)+\log\binom{m^2+m}{m}+\log(m^2+1).
\end{equation*}
\end{lem}
Now we want to bound the norm of the  conjugation of $\ub$ by $\vb^{-1}$ in terms of $\ub$ and $\vb$.
\begin{lem}
\label{R6}
Let $\ub,\vb\in \Lie(\Gl_m(\Com))$ with $\vb$ an invertible matrix and let $\vb^{-1}$ be the inverse matrix of $\vb$. Then
\begin{equation*}
\|\vb^{-1}\star \ub\star \vb\|_{\Bc_0}\leq (m+1)!\frac{\|\vb\|^m_{\Bc_0}}{\det(\vb)}\|\ub\|_{\Bc_0}.
\end{equation*}
\end{lem}
\begin{proof}
Write $\ub=\sum_{k=1}^{m^2}z_k\fb_k$, $\vb=\sum_{k=1}^{m^2}y_k\fb_k$ and $\vb^{-1}=\sum_{k=1}^{m^2}y'_k\fb_k$. The Cauchy-Schwarz inequality yields
\begin{align*}
\bigg|\sum_{1\leq r,s\leq m}y'_{(i-1)m+r}y_{(s-1)m+j}z_{(r-1)m+s}\bigg|^2&\leq\\
m^2\sum_{1\leq r,s\leq m}\big|y'_{(i-1)m+r}y_{(s-1)m+j}z_{(r-1)m+s}\big|^2&
\end{align*}
and thereby
\begin{align}
\label{E5}
\|\vb^{-1}\star \ub\star \vb\|^2_{\Bc_0}&=\sum_{1\leq i,j\leq m}\bigg|\sum_{1\leq r,s\leq m}y'_{(i-1)m+r}y_{(s-1)m+j}z_{(r-1)m+s}\bigg|^2\nonumber\\
&\leq\sum_{1\leq i,j\leq m}m^2\bigg(\sum_{1\leq r,s\leq m}\big|y'_{(i-1)m+r}y_{(s-1)m+j}z_{(r-1)m+s}\big|^2\bigg)\nonumber\\
&=m\sum_{1\leq i,j\leq m}\sum_{1\leq r,s\leq m}\big|y'_{(i-1)m+r}y_{(s-1)m+j}\big|^2\big|z_{(r-1)m+s}\big|^2\nonumber\\
&\leq m \|\ub\|^2_{\Bc_0}\sum_{1\leq i,j,r,s\leq m}\big|y'_{(i-1)m+r}y_{(s-1)m+j}\big|^2. 
\end{align}
Now see that
\begin{align}
\label{E6}
\sum_{1\leq i,j,r,s\leq m}\big|y'_{(i-1)m+r}y_{(s-1)m+j}\big|^2= \|\vb\|^2_{\Bc_0}\|\vb^{-1}\|^2_{\Bc_0}
\end{align}
The scalars $y'_1,\ldots,y'_{m^2}$ are the cofactors of $\vb$ divided by $\det(\vb)$. Hence  there are homogeneous polynomials $Q_1,\ldots,Q_{m^2}\in \Zet[x_1,\ldots,x_{m^2}]$ with coefficients $\pm 1$ and degree $m-1$ such that
\begin{equation*}
\frac{Q_k(y_1,\ldots, y_{m^2})}{\det(\vb)}=y'_k.
\end{equation*}
For all $k\in\{1,\ldots,m^2\}$, a trivial calculation gives
 \begin{equation*}
  \big|Q_k(y_1,\ldots, y_{m^2})\big|\leq (m-1)!\|\vb\|^{m-1}_{\Bc_0}
 \end{equation*}
and this inequality implies that  
\begin{align}
\label{E7}
\|\vb^{-1}\|^2_{\Bc_0}&=\sum_{k=1}^{m^2}\bigg|\frac{Q_k(y_1,\ldots, y_{m^2})}{\det(\vb)}\bigg|^2\nonumber\\
&\leq \bigg(m!\frac{\|\vb\|^{m-1}_{\Bc_0}}{\det(\vb)}\bigg)^2.
\end{align}
Finally, the claim is a consequence (\ref{E5}), (\ref{E6}) and (\ref{E7}).
\end{proof}
\section{Main case}
In this section we demonstrate Theorem \ref{R1} in the case where:
\begin{enumerate}
\item[$\bullet$]$G=\Gl_m$.
\item[$\bullet$]$W$ is a hyperplane.
\item[$\bullet$]$\Bc$ is the basis $\Bc_0$.
\end{enumerate}
Nevertheless,  this particular case is extremely important in the proof of Theorem \ref{R1}. We explain the main idea of the proof of this particular case. First it is shown, using the Jordan canonical form and the invariance under conjugation of the trace operator, that there are $\lambda_1,\ldots,\lambda_m\in\Com$ and  $\beta_0,\ldots,\beta_m\in\field$ depending on $\ub$ and $W$ such that $\exp_{\Com}(\lambda_i)\in\field$ for all $i\in\{1,\ldots,m\}$  and 
\begin{equation*}
\dist_{\Bc_0}(\ub,W)=\bigg|\beta_0+\sum_{i=1}^m\beta_i\lambda_i\bigg|. 
\end{equation*}
Then, using \cite[Thm. 9.1 \& Prop. 9.21]{Waldschmidt}, we find a lower bound of $\big|\beta_0+\sum_{i=1}^l\beta_i\lambda_i\big|$. To conclude the proof, it remains to express this lower bounds in terms of $\ub$ and $W$.

 We stated that \cite[Thm. 9.1 \& Prop. 9.21]{Waldschmidt} its a crucial tool in this section. Thus,  before we state the main result of this section, we present the following  proposition which is a straight consequence of  \cite[Thm. 9.1 \& Prop. 9.21]{Waldschmidt}.
\begin{prop}
\label{R7}
Let $\lambda_1,\ldots,\lambda_m\in\Com$ be  such that $\alpha_i:=\exp_{\Com}(\lambda_i)\in\Ratcl$ for all $i\in\{1,\ldots,m\}$.  Let $\beta_0,\ldots,\beta_m\in\Ratcl$ satisfying that $\beta_0+\sum_{i=1}^m\beta_i\lambda_i\neq 0$ and let $\field$ be a finite extension of $\Rat$ such that $\Rat(\alpha_1,\ldots\alpha_m,\beta_0,\ldots,\beta_m)\subseteq \field$. Define  $a_1,\ldots,a_m,b,c$ as follows
\begin{align*}
a_i&:=\max\bigg\{\hi(\alpha_i),\frac{e|\lambda_i|}{[\field:\Rat]},\frac{1}{[\field:\Rat]}\bigg\},\\
c&:=\max\big\{\log([\field:\Rat]),1\big\},\\
b&:=\max\bigg\{\exp_\Com(c),\max_{1\leq i\leq m}[\field:\Rat]a_i,\exp_\Com\Big(\max_{0\leq i\leq m}\hi(\beta_i)\Big),e\bigg\}.
\end{align*}
Then
 \begin{equation*}
 \log\Bigg(\bigg|\beta_0+\sum_{i=1}^m\beta_i\lambda_i\bigg|\Bigg)\geq -2^{26m}m^{3m}[\field:\Rat]^{m+2}\log(b)\bigg(\prod_{i=1}^ma_i\bigg)c
 \end{equation*}
\end{prop}
 \begin{proof}
  Write $\mathbb{L}:=\Rat(\alpha_1,\ldots\alpha_m,\beta_0,\ldots,\beta_m)$ and define the real numbers $\Lambda,E,\\\log(A_1),\ldots, \log(A_m), \log(E^*)$ and $B$ as follows
   \begin{align*}
   \Lambda&:=\beta_0+\sum_{i=1}^m\beta_i\lambda_i,\\
   E&:=e,\\
   \log(A_i)&:=\max\bigg\{\hi(\alpha_i),\frac{e|\lambda_i|}{[\mathbb{L}:\Rat]},\frac{1}{[\mathbb{L}:\Rat]}\bigg\},\\
\log(E^*)&:=\max\big\{\log([\mathbb{L}:\Rat]),1\big\},\\
B&:=\max\bigg\{\exp_\Com(c),\max_{1\leq i\leq m}[\mathbb{L}:\Rat]\log(A_i),\exp_\Com\Big(\max_{0\leq i\leq m}\hi(\beta_i)\Big),e\bigg\}.
\end{align*}. Since $\Lambda\neq 0$ and 
 \begin{equation*}
 [\mathbb{L}:\Rat]^3\log(B)\log(A_i)\log(E^*)\geq \log([\mathbb{L}:\Rat])\log(E)^2 \qquad\forall i\in\{1,\ldots,m\},
 \end{equation*}
 we have by \cite[Prop. 9.21]{Waldschmidt} that the claim of \cite[Thm. 9.1]{Waldschmidt} is true even if $\lambda_1,\ldots,\lambda_m$ are not $\Rat-$linearly independent. Thus we apply \cite[Thm. 9.1]{Waldschmidt} and we get that  
 \begin{equation*}
 \log(|\Lambda|)\geq -2^{26m}m^{3m}[\mathbb{L}:\Rat]^{m+2}\log(B)\bigg(\prod_{i=1}^m\log(A_i)\bigg)\log(E^*)\log(E)^{-m-1}.
 \end{equation*}
 Finally this proposition is true  insomuch as
 \begin{align*}
2^{26m}m^{3m}[\mathbb{L}:\Rat]^{m+2}\log(B)\bigg(\prod_{i=1}^m\log(A_i)\bigg)\log(E^*)\log(E)^{-m-1}&\leq\\
2^{26m}m^{3m}[\field:\Rat]^{m+2}\log(b)\bigg(\prod_{i=1}^ma_i\bigg)c.
 \end{align*} 
 \end{proof}
Recall that for each $\ub\in \Lie(G(\Com))$ we have that  $\jb_\ub$ is the Jordan canonical form of $\ub$ and  $C_\ub(\field)$ is the subset of  invertible elements $\vb$ of $\Lie(\Gl_m(\Com))$ (as matrices) with entries in $\field$ such that $\vb\star \jb_\ub\star \vb^{-1}=\ub$. Also remember that $\ub$ is a $\field-$point if $\exp_{G(\Com)}(\ub)\in\field$ and the characteristic polynomial of $\exp_{G(\Com)}(\ub)$ (as an element of $\Gl_m$) splits in $\field$.

\begin{prop}
\label{R8}
 Let $W$ be an hyperplane of  $\Lie(\Gl_m(\Com))$ defined over $\field$ and  $\ub\in\Lie(\Gl_m(\Com))\setminus W$ a $\field$-point. Define
\begin{align*}
b_4:=&b_2(\ub)\max\big\{\h(\exp_{\Gl_m(\Com)}(\ub)),\|\ub\|_{\Bc_0},e\big\}\\
c_5:=&2^{32m+24}m^{m^2+8m+13}[\field:\Rat]^{m+5}.
\end{align*} 
 Then 
\begin{equation*}
\log(\dist_{\Bc_0}(\ub,W))\geq -c_5\log(b_4)b_4^{m+1}\max\{1,\h_{\Bc_0}(W)\}.
\end{equation*}
\end{prop}
\begin{proof}
 Inasmuch as $W$ is an hyperplane, there are $y_1,\ldots,y_{m^2}\in\field$ satisfying:
 \begin{enumerate}
 \item[i)]$W=\Bigg\{\sum_{k=1}^{m^2}z_k\fb_k\in\Lie(\Gl_m(\Com)):\;\sum_{k=1}^{m^2}y_k\overline{z_k}=0\Bigg\}$.
\item[ii)]There is $k\in\{1,\ldots,m^2\}$ such that $y_k=1$.
 \end{enumerate}
We assume without loss of generality that $y_1=1$. Let $t_1,\ldots,t_{m^2}\in\Com$ be the scalar which satisfy $\ub=\sum_{k=1}^{m^2}t_k\fb_k$. Set
\begin{align*}
\wb&:=\sum_{1\leq i,j\leq m}\overline{y_{(j-1)m+i}} \fb_{(i-1)m+j}.
\end{align*}
Let $\vb\in C_\ub(\field)$ and write
\begin{equation*}
b_5(\vb):=\max\bigg\{e,\hi(\vb),\frac{\|\vb\|^m_{\Bc_{\Gl_m}}}{\det(\vb)}\bigg\}.
\end{equation*}
Let $s_1,\ldots,s_{m^2}\in\field$ be the scalars such that
\begin{equation*}
\vb^{-1}\star \wb\star \vb=\sum_{1\leq i,j\leq m}s_{(j-1)m+i}\fb_{(i-1)m+j}.
\end{equation*} 
Let $r_1,\ldots,r_{m^2}\in\field$ be scalars such that $\jb_\ub=\sum_{k=1}^{m^2}r_k \fb_k$. Since $\jb_\ub$ is a Jordan matrix,  if $r_k\neq 0$ then $k=(i-1)m+i $ for some $i\in\{1,\ldots,m\}$ or $k=(i-m-1)m+i-m+1$ for some $i\in\{m+1,\ldots,2m-1\}$. Moreover, $r_{(i-m-1)m+i-m+1}\in\{0,1\}$ for all $i\in\{m+1,\ldots,2m-1\}$. We abbreviate the notation as follows: for all $i\in\{1,\ldots,m\}$
\begin{equation*}
 \lambda_i:=r_{(i-1)m+i}\qquad \beta_i:=s_{(i-1)m+i}, 
 \end{equation*} and
 \begin{equation*}
  \beta_0:=\sum_{i=m+1}^{2m-1}r_{(i-m-1)m+i-m+1}s_{(i-m-1)m+i-m+1}. 
 \end{equation*}
 Thus 
 \begin{equation}
 \label{E8}
\sum_{k=1}^{m^2}r_ks_k=\beta_0+\sum_{i=1}^{m}\beta_i\lambda_i.
 \end{equation}
  Since the trace operator $\Tr$ is invariant under matrix conjugation, we conclude that 
\begin{equation}
\label{E9}
\Tr(\ub\star \wb)=\Tr\big(\vb^{-1}\star(\ub\star \wb)\star \vb\big)=\Tr\big(\jb_\ub\star \vb^{-1}\star \wb\star \vb\big).
\end{equation}
Set $\wb':=\sum_{k=1}^{m^2}y_k\fb_k$.  Since $\|\wb\|_{\Bc_0}=\|\wb'\|_{\Bc_0}$, see that  
\begin{equation*}
\dist_{\Bc_0}(\ub,W)=\frac{\Big\vert\sum_{k=1}^{m^2} y_k\overline{t_k}\Big\vert}{\|\wb\|_{\Bc_0}}
\end{equation*}
and therefore
\begin{align}
\label{E10}
\|\wb\|_{\Bc_0} \cdot \dist_{\Bc_0}(\ub,W)=&\Bigg\vert\sum_{k=1}^{m^2} y_k\overline{t_k}\Bigg\vert\nonumber\\
=&\Bigg\vert\sum_{k=1}^{m^2} t_k\overline{y_k}\Bigg\vert\nonumber\\
=&\big\vert\Tr(\ub\star \wb)\big\vert\nonumber\\
=&\big\vert\Tr(\jb_\ub\star \vb^{-1}\star \wb\star \vb)\big\vert&\text{by (\ref{E9})}\nonumber\\
=&\Bigg\vert\sum_{k=1}^{m^2}r_ks_k\Bigg\vert\nonumber\\
=&\Bigg\vert\beta_0+\sum_{i=1}^{m}\beta_i\lambda_i\Bigg\vert&\text{by (\ref{E8})}.
\end{align}
  Since $\ub\not\in W$, we have by (\ref{E10}) that $\beta_0+\sum_{i=1}^{m}\beta_i\lambda_i\neq 0$.  Inasmuch as $\jb_\ub$ is the Jordan matrix of $\ub$, for all $i\in\{1,\ldots,m\}$  the complex numbers $\exp_{\Com}(\lambda_i)$  are the eigenvalues of $\exp_{\Gl_m(\Com)}(\ub)$ and consequently they are in $\field$. Insomuch as the entries of $\vb$ and $\wb$ are in $\field$, we get that $\beta_i\in\field$ for all $i\in\{0,\ldots,m\}$ and thus the assumptions of  Proposition \ref{R7} are fulfilled.  For each $j\in\{1,\ldots,m\}$, set 
\begin{align*}
c_6&:=\max\bigg\{1,\log([\field:\Rat])\bigg\}\nonumber\\
b_{j,6}&:=\max\bigg\{\h(\exp_{\Com}(\lambda_j)),\frac{e|\lambda_j|}{[\field:\Rat]},\frac{1}{[\field:\Rat]}\bigg\}\nonumber\\
b_7&:=\max\bigg\{c_6,[\field:\Rat]\max_{1\leq i\leq m}b_{i,6},\exp_\Com\Big(\max_{0\leq i\leq m}\hi(\beta_i)\Big),e\bigg\}.
\end{align*}
Hence Proposition \ref{R7} and (\ref{E10}) yield
\begin{equation}
\label{E11}
\log\big(\|\wb\|_{\Bc_0} \cdot \dist_{\Bc_0}(\ub,W)\big)\geq -2^{26m}m^{3m}[\field:\Rat]^{m+2}c_6\log(b_7)\prod_{j=1}^mb_{j,6}.
\end{equation}

Now we find a lower bound of $\log\big(\frac{1}{\|\wb\|_{\Bc_0}}\big)$ in terms of $\max\{1,\h_{\Bc_0}(W)\}$.   Lemma \ref{R2} implies that
\begin{equation*}
|\h_{\Bc_0}\big(\big<\wb'\big>\big)-\h_{\Bc_0}(W)|\leq \log\binom{m^2}{1}
\end{equation*}
and, since $\h_{\Bc_0}\big(\big<\wb\big>\big)=\h_{\Bc_0}\big(\big<\wb'\big>\big)$, we conclude that 
\begin{equation}
\label{E12}
\max\Big\{1,\h_{\Bc_0}\big(\big<\wb\big>\big)\Big\}\leq 4\max\{\log(m),1\}\max\{1,\h_{\Bc_0}(W)\}.
\end{equation}
Insomuch as $y_1=1$, we have that
\begin{align*}
\h_{\Bc_0}\big(\big<\wb\big>\big)&\geq \frac{1}{[\field:\Rat]}\log\Big(\max_{1\leq k\leq m^2}|y_i|\Big)\\
&\geq \frac{1}{[\field:\Rat]}\log\bigg(\frac{\|\wb\|_{\Bc_0}}{m}\bigg)
 \end{align*} and then
\begin{align}
\label{E28}
\log\bigg(\frac{1}{\|\wb\|_{\Bc_0}}\bigg)&\geq 2\max\{\log(m),1\}[\field:\Rat]\max\Big\{1,\h_{\Bc_0}\big(\big<\wb\big>\big)\Big\}
\end{align}
From (\ref{E12}) and (\ref{E28}) 
\begin{equation}
\label{E13}
\log\bigg(\frac{1}{\|\wb\|_{\Bc_0}}\bigg)\geq -8\max\{\log(m),1\}^2[\field:\Rat]\max\{\h_{\Bc_0}(W),1\}.
\end{equation}

Lets find an upper bound of  $\max_{0\leq i\leq m}\hi(\beta_i)$  in terms of $\max\{\h_{\Bc_0}(W),1\}$. Insomuch as $r_{(i-m-1)m+i-m+1}\in\{0,1\}$ for all $i\in\{m+1,\ldots,2m-1\}$, a standard height bound calculation (see for instance \cite[Thm. B.2.5]{Hindry-Silverman}) implies that
\begin{equation*}
\hi([\beta_0:\ldots:\beta_m])\leq \hi([s_1:\ldots:s_{m^2}])+\log(2m-1)
\end{equation*}
and then
\begin{align}
\label{E14}
\max_{0\leq i\leq m}\hi(\beta_i)&\leq \hi([\beta_0:\ldots:\beta_m])\nonumber\\
&\leq\hi([s_1:\ldots:s_{m^2}])+\log(2m-1)\nonumber\\
&=\hi\big(\vb^{-1}\star \wb\star \vb\big)+\log(2m-1).
\end{align}
Write
\begin{equation*}
c_7:=3m\max\bigg\{\log\binom{m^2+m}{m}+\log(m^2+1),1\bigg\}
\end{equation*}
 so Lemma \ref{R5} implies that
\begin{align}
\label{E15}
\hi\big(\vb^{-1}\star \wb\star \vb\big)&\leq c_7b_5(\vb)\max\{1,\hi(\wb)\}\nonumber\\
&=c_7b_5(\vb)\max\Big\{1,\h_{\Bc_0}\big(\big<\wb\big>\big)\Big\}&\text{since $y_1=1$.}
\end{align} 
Calling
\begin{equation*}
c_8:=8c_7\max\{\log(2m-1),1\}\max\{\log(m),1\},
\end{equation*}
we conclude from  (\ref{E14}), (\ref{E15}) and (\ref{E12}) that 
\begin{equation}
\label{E16}
\max_{0\leq i\leq m}\hi(\beta_i)\leq c_8b_5(\vb)\max\{1,\h_{\Bc_0}(W)\}.
\end{equation}

Write
\begin{equation*}
b_8:=\max\big\{\h(\exp_{\Gl_m(\Com)}(\jb_\ub)),\|\jb_\ub\|_{\Bc_0},e\big\}
\end{equation*}
For all $j\in\{1,\ldots,m\}$,  we shall find upper bounds of $b_{j,6}$ and $\log(b_7)$ in terms of $b_8$. Recall that for each  $j\in\{1,\ldots,m\}$ we have that $\exp_{\Com}(\lambda_j)$ is an eigenvalue of $\exp_{\Gl_m}(\jb_\ub)$ so 
\begin{equation}
\label{E17}
b_{j,6}\leq b_8.
\end{equation} 
Moreover, setting
\begin{equation*}
c_9:=\max\{c_6,c_8,2\log([\field:\Rat]),e\}
\end{equation*}
we get from (\ref{E16}) and (\ref{E17})
\begin{equation}
\label{E18}
\log(b_7)\leq c_9b_5(\vb)\log(b_8)\max\{\h_{\Bc_0}(W),1\}
\end{equation}

Define 
\begin{align*}
b_9&:=\max\big\{\h(\exp_{\Gl_m(\Com)}(\ub)),\|\ub\|_{\Bc_0},e\big\}
\end{align*}
and the next step is to give an upper bound of $b_8$ in terms of $b_9$. If $\vb=\sum_{i=1}^{m^2}l_i\fb_i$ and $l:=[1:l_1:\ldots:l_{m^2}]$, then $\h(l)=\hi(\vb)$. Furthermore
\begin{equation*}
\exp_{\Gl_m(\Com)}(\jb_\ub)=\exp_{\Gl_m(\Com)}\big(\vb^{-1}\star\ub\star\vb\big)=l^{-1}\exp_{\Gl_m(\Com)}(\ub)l
\end{equation*}
and therefore Lemma \ref{R4} leads to
\begin{equation}
\label{E19}
\h\big(\exp_{\Gl_m(\Com)}(\jb_\ub)\big)\leq \h(\exp_{\Gl_m(\Com)}(\ub))+ m\hi(\vb)+\log\binom{m^2+m}{m}+\log(m^2+1).
\end{equation}
Now Lemma \ref{R6} implies that
\begin{equation}
\label{E20}
\|\jb_\ub\|_{\Bc_0}\leq (m+1)!\frac{\|\vb\|^m_{\Bc_0}}{\det(\vb)}\|\ub\|_{\Bc_0}.
\end{equation}
Setting
\begin{equation*}
c_{10}:=3\max\bigg\{(m+1)!,\log\binom{m^2+m}{m}+\log(m^2+1)\bigg\},
\end{equation*}
the inequalities (\ref{E19}) and (\ref{E20}) give
\begin{equation}
\label{E21}
b_8\leq c_{10}b_5(\vb)b_9.
\end{equation}

We conclude the proof. Set
\begin{align*}
c_{11}&:=2^{26m+3}m^{3m+2}[\field:\Rat]^{m+3}c_6.
\end{align*}
Finally
\begin{align*}
\log\big(\dist_{\Bc_0}(\ub,W)\big)\geq& -2^{26m}m^{3m}[\field:\Rat]^{m+2}c_6\log(b_7)\prod_{j=1}^mb_{j,6}\\
&+\log\bigg(\frac{1}{\|\wb\|_{\Bc_0}}\bigg)&\text{by (\ref{E11})}\\
\geq&-c_{11}\max\{\h_{\Bc_0}(W),1\}-c_{11}\log(b_7)\prod_{j=1}^mb_{j,6}&\text{by (\ref{E13})}\\
\geq& -c_{11}\max\{\h_{\Bc_0}(W),1\}-c_{11}\log(b_7)b_8^m&\text{by (\ref{E17})}\\
\geq& -c_{11}\max\{\h_{\Bc_0}(W),1\}\\
&-c_9c_{11}b_5(\vb)\log(b_8)b_8^m\max\{\h_{\Bc_0}(W),1\}&\text{by (\ref{E18})}\\
\geq& -c_5\log\big(b_5(\vb)b_9\big)(b_5(\vb)b_9)^{m+1}\max\{\h_{\Bc_0}(W),1\}&\text{by (\ref{E21})}
\end{align*}
 and the statements follows since $b_4=\inf_{\vb\in C_\ub(\field)}b_5(\vb)b_9$
\end{proof}
\section{Proof of Theorem \ref{R1}}
In this section we demonstrate Theorem \ref{R1}. At the end of this section, we make some remarks. 
\begin{proof}
Set  
\begin{equation*}
W^\perp:=\Bigg\{\sum_{k=1}^{m^2}z_k\fb_k\in\Lie(\Gl_m(\Com)):\;\sum_{k=1}^{m^2}y_k\overline{z_k}=0 \text{ for all }\sum_{k=1}^{m^2}y_k\fb_k\in W\Bigg\}.
\end{equation*}
Let $\Delta_\field$ be the discriminant of $\field$ and define
\begin{equation*}
c_{12}:=2\max\bigg\{\frac{m^2-d}{2[\field:\Rat]}\log\Bigg(\bigg(\frac{2}{\pi}\bigg)\vert\Delta_\field\vert\Bigg)+\frac{1}{2}\log\binom{m^2}{m^2-d},1\bigg\}.
\end{equation*}
As a consequence of  \cite[Thm. 9]{Bombieri-Vaaler}, there are $\wb_j:=\sum_{i=1}^{m^2}y_{i,j}\fb_i$ with $y_{i,j}\in\field$ for all $j\in\{1,\ldots,m^2-d\}$ and $i\in\{1,\ldots,m^2\}$ such that
\begin{enumerate}
\item[i)]$\{\wb_1,\ldots,\wb_{m^2-d}\}$ is a basis of $W^\perp$.
\item[ii)]$\sum_{i=1}^{m^2-d}\h_{\Bc_0}(\wb_i)\leq c_{12}\max\big\{\h_{\Bc_0}(W^\perp),1\big\}$.
\end{enumerate}
For each $j\in\{1,\ldots,m^2-d\}$, set
\begin{equation*}
W_j:=\Bigg\{\sum_{i=1}^{m^2}z_i\fb_i\in\Lie(\Gl_m(\Com)):\;\sum_{i=1}^{m^2}y_{i,j}\overline{z_i}=0 \Bigg\}
\end{equation*}
and
\begin{align*}
c_{13}&:=2\max\bigg\{(m^2-d)\log(m^2),\log\binom{m^2}{d},1\bigg\}
\end{align*}
Since $W=\bigcap_{i=1}^{m^2-d}W_i$, there is $W_i$ such that $\ub\not\in W_i$; assume without loss of generality that $\ub\not\in W_1$. Then 
\begin{align}
\label{E22}
\h_{\Bc_0}(W_1)&\leq \h_{\Bc_0}\big(\big<\wb_1\big>\big)+\log(m^2)&\text{by Lemma \ref{R2}}\nonumber\\
&\leq c_{12}c_{13}\max\big\{\h_{\Bc_0}(W^\perp),1\big\}\nonumber\\
&\leq c_{12}c^2_{13}\max\big\{\h_{\Bc_0}(W),1\big\}&\text{by Lemma \ref{R2}}
\end{align}
Thus $W_1$ is an hyperplane of $\Lie(\Gl_m(\Com))$ defined over $\field$ with respect to $\Bc_0$ and we may apply  Proposition \ref{R8} to get 
\begin{equation}
\label{E23}
\log(\dist_{\Bc_0}(\ub,W_1))\geq -c_5\log(b_4)b_4^{m+1}\max\{\h_{\Bc_0}(W_1),1\}.
\end{equation}
Let
\begin{equation*}
c_{14}:=\max\Bigg\{\sqrt{n\sum_{i=1}^{m^2}\sum_{j=1}^{n} |z_{i,j}|^2},e\Bigg\}.
\end{equation*}
For all $\vb\in \Lie(G(\Com))\subseteq \Lie(\Gl_m(\Com))$, the Cauchy-Schwarz inequality yields
\begin{equation}
\label{E24}
\|\vb\|_{\Bc_0}\leq c_{14} \|\vb\|_{\Bc}.
\end{equation}
Hence 
\begin{equation}
\label{E25}
b_4\leq c_{14}^{m+1}b_{3}.
\end{equation}
Also (\ref{E24}) implies that 
\begin{align}
\label{E26}
\log(\dist_{\Bc_0}(\ub,W_1))&\leq \log(\dist_{\Bc_0}(\ub,W_1\cap \Lie(G(\Com))))\nonumber\\
&\leq \log(\dist_{\Bc_0}(\ub,W))\nonumber\\
&\leq \log(\dist_{\Bc}(\ub,W))+\log(c_{14})
\end{align}
Finally, set
\begin{align*}
c_{15}&:=2c_5c_{12}c^2_{13}c_{14}^{m^2+2m+2}
\end{align*}
and therefore
\begin{align*}
\log(\dist_{\Bc}(\ub,W))&\geq \log(\dist_{\Bc_0}(\ub,W_1))-\log(c_{14})&\text{by (\ref{E26})}\nonumber\\
&\geq -c_5\log(b_4)b_4^{m+1}\max\{\h_{\Bc_0}(W_1),1\}-\log(c_{14})&\text{by (\ref{E23})}\nonumber\\
&\geq -c_5c_{14}^{m^2+2m+2}\log(b_{3})b_{3}^{m+1}\max\{\h_{\Bc_0}(W_1),1\}&\text{by (\ref{E25})}\nonumber\\
&\geq -c_{15}\log(b_{3})b_{3}^{m+1}\max\{\h_{\Bc_0}(W),1\}&\text{by (\ref{E22})}\nonumber\\
&\geq -c_4\log(b_{3})b_{3}^{m+1}\max\{\h_{\Bc}(W),1\}&\text{by Lemma \ref{R3}.}
\end{align*} 
\end{proof}
We conclude remarking some facts.
\begin{rem}
\label{R9}
 Let $W$ be a $d-$dimensional linear subspace of  $\Lie(G(\Com))$ defined over $\Ratcl$ with respect to $\Bc$ and let  $\ub\in\Lie(G(\Com))\setminus W$ be such that $\exp_{G(\Com)}(\ub)\in G(\Ratcl)$. Then there exists a finite extension $\field$ of $\Rat$ such that $G$ is defined over $\field$, $W$ is defined over $\field$ with respect to $\Bc$, $\exp_{G(\Com)}(\ub)\in G(\field)$ and the characteristic polynomial of $\exp_{G(\Com)}(\ub)$ splits in $\field$. Then 
 \begin{equation*}
\log(\dist_\Bc(\ub,W))\geq -c_4\log(b_3)b^{m+1}_3\max\{\h_\Bc(W),1\}.
\end{equation*}
\end{rem} 
\begin{rem}
\label{R10}The lower bound of Theorem \ref{R1} with respect to $\max\{1,\h_\Bc(W)\}$ is optimal in the following sense.
Take $\field=\Rat$,  $G=\Gl_2 $ and  $\Bc=\Bc_0$. For all $k\in\Nat$,  the following hyperplanes of $\Lie(G(\Com))$
\begin{equation*}
W_k:=\left\{\left( \begin{array}{cc}
y_1 & y_2 \\
y_3 & y_4 \\
\end{array} \right):\;y_1,\ldots,y_4\in \Com,\;y_1+\frac{y_4}{k}=0\right\}
\end{equation*}
are defined over $\Rat$. Define
 \begin{align*}
\vb&:=\left( \begin{array}{cc}
0 & 0 \\
0 & \log(2) \\
\end{array} \right).
\end{align*}
Thus for all $k\in\Nat$
\begin{equation}
\label{E27}
\log(\dist_\Bc(\vb,W_k))=\log\bigg(\frac{\log(2)}{\sqrt{k^2+1}}\bigg)>\log(\log(2))-\log(k+1)\quad\text{and}\quad \h_\Bc(W_k)=\log(k).
\end{equation}
Let  $\phi:\Real\rightarrow \Real$ be a function  satisfying that
\begin{equation*}
\log(\dist_\Bc(\ub,W))\geq -\phi(\max\{\h_\Bc(W),1\})
\end{equation*}
for all proper linear subspaces  $W$ of  $\Lie(G(\Com))$ defined over $\Rat$ with respect to $\Bc$ and $\ub\in\Lie(G(\Com))$ a $\Rat-$point. Then (\ref{E27}) implies that $\lim_{x\rightarrow\infty}\frac{\phi(x)}{x}\geq 1$.

\end{rem}
\begin{rem}
\label{R11}
The term   $b_2(\ub)$ in Theorem \ref{R1} cannot be excluded as we show in the following example.  Take $\field=\Rat$, $G=\Gl_2 $ and  $\Bc=\Bc_0$. The following hyperplane of $\Lie(G(\Com))$
\begin{equation*}
W:=\left\{\left( \begin{array}{cc}
y_1 & y_2 \\
y_3 & y_4 \\
\end{array} \right):\;y_1,\ldots,y_4\in \Com,\;y_4=0\right\}
\end{equation*}
is defined over $\Rat$. For $k\in\Nat\setminus\{1\}$, define 
\begin{align*}
\ub_k&:=\left( \begin{array}{cc}
1+\frac{1}{k} & \frac{1}{k} \\
-\frac{1}{k} & 1-\frac{1}{k} \\
\end{array} \right)\star\left( \begin{array}{cc}
2\pi i & 0 \\
0 & 0 \\
\end{array} \right)\star\left( \begin{array}{cc}
1-\frac{1}{k} & -\frac{1}{k} \\
\frac{1}{k} & 1+\frac{1}{k} \\
\end{array} \right)\\
&=\left( \begin{array}{cc}
2\pi i\bigg(1-\frac{1}{k^2}\bigg) & -2\pi i\frac{1}{k}\bigg(1+\frac{1}{k}\bigg) \\
-2\pi i\frac{1}{k}\bigg(1-\frac{1}{k}\bigg) & 2\pi i \frac{1}{k^2}\\
\end{array} \right);\\
\end{align*}
 hence $\exp_{G(\Com)}(\ub_k)=\left( \begin{array}{cc}
1 & 0 \\
0 & 1 \\
\end{array} \right)$ for all $k\in\Nat$. For any fixed constants  $c,\kappa>0$,  there is  $K(c,\kappa)\in\Nat$ such that  for all $k>K(c,\kappa)$ 
\begin{align*}
\log(\dist_{\Bc}(\ub_k,W))&=\log\bigg(\bigg\vert2\pi i \frac{1}{k^2}\bigg\vert\bigg)\\
&<-100^\kappa c\max\{1,\h_{\Bc}(W)\}\\
&<-c\max\big\{\h(\exp_{G(\Com)}(\ub_k)),\|\ub_k\|_{\Bc},1\big\}^{\kappa}\max\{1,\h_{\Bc}(W)\}.
\end{align*}
\end{rem}
 
\end{document}